\documentclass[11pt,reqno]{amsart}
\usepackage[T1]{fontenc}
\usepackage{lmodern}
\usepackage[margin=1.05in]{geometry}
\usepackage{amsmath,amssymb,amsthm,mathtools}
\usepackage{microtype}
\usepackage{enumitem}
\usepackage[dvipsnames]{xcolor}
\usepackage[colorlinks=true,linkcolor=MidnightBlue,citecolor=MidnightBlue,urlcolor=MidnightBlue]{hyperref}
\hypersetup{pdftitle={Bounded symbols for truncated Toeplitz operators: the one-component characterization},
 pdfsubject={Bounded symbols for truncated Toeplitz operators},
 pdfauthor={Pan Ma}}
\newtheorem{theorem}{Theorem}[section]
\newtheorem{proposition}[theorem]{Proposition}
\newtheorem{lemma}[theorem]{Lemma}
\newtheorem{corollary}[theorem]{Corollary}
\theoremstyle{definition}

\theoremstyle{remark}
\newtheorem{remark}[theorem]{Remark}
\numberwithin{equation}{section}
\newcommand{\D}{\mathbb D}
\newcommand{\T}{\mathbb T}

\newcommand{\abs}[1]{\lvert #1\rvert}
\newcommand{\norm}[1]{\lVert #1\rVert}
\newcommand{\ip}[2]{\langle #1,#2\rangle}

\newcommand{\TT}{\mathcal T}
\newcommand{\Sone}{\mathcal S_1}
\DeclareMathOperator{\dist}{dist}
\DeclareMathOperator{\supp}{supp}
\DeclareMathOperator{\Rea}{Re}
\allowdisplaybreaks[2]
\setlist[enumerate]{label=\textup{(\roman*)},leftmargin=2.2em}
\title[Bounded Symbols and One-Component Inner Functions]
      {Bounded Symbols and One-Component Inner Functions}
\author{Pan Ma}
\address{School of Mathematics and Statistics, HNP-LAMA,
Central South University, Changsha 410083, China}
\email{pan.ma@csu.edu.cn}

\begin{document}
\begin{abstract}
	Let $\theta$ be a nonconstant inner function, we prove that every bounded truncated Toeplitz operator on the model
	space $K_\theta=H^2\ominus\theta H^2$ admits a bounded symbol if and
	only if $\theta$ is one-component, that
	is, if $\{z\in\mathbb D:|\theta(z)|<r\}$ is nonempty and connected for some
	$r\in(0,1)$. This establishes the previously open necessary direction
	and resolves a conjecture of Baranov, Chalendar, Fricain, Mashreghi,
	and Timotin. For every non-one-component inner function, we construct
	a sequence of norm-one rank-one truncated Toeplitz operators whose
	infimal $L^\infty$-norms over all representing symbols tend to infinity.
	The quantitative lower bound is obtained by forming two inner divisors from angularly separated parts of the zero
	set and singular measure. The argument applies without restrictions on the
	Blaschke or singular parts of $\theta$. The order of this lower bound
	is sharp, already for a family of degree-two Blaschke products. A continuation argument for local inverse branches produces the
	required sequence of points. Consequently, every non-one-component model
	space supports a trace-class truncated Toeplitz operator without a
	bounded symbol. Combined with known equivalences, the characterization
	also yields geometric criteria for weak factorization and for
	continuous-symbol representations of compact truncated Hankel
	operators.
\end{abstract}
\maketitle

\section{Introduction and main results}\label{sec:intro}

\subsection{Background and main results}

Let $\theta$ be a nonconstant inner function on the unit disk $\D$, and let
$K_\theta=H^2\ominus\theta H^2$ be its model space. If $P_\theta$ denotes
the orthogonal projection from $L^2(\T)$ onto $K_\theta$, a symbol
$\varphi\in L^2(\T)$ defines the truncated Toeplitz operator
\begin{equation}\label{intro:tto}
 A_\varphi^\theta f=P_\theta(\varphi f),\qquad
 f\in K_\theta\cap H^\infty.
\end{equation}
If $A_\varphi^\theta$ extends boundedly to $K_\theta$, the same
notation denotes its bounded extension, and
$\TT(\theta)$ denotes the space of all such operators. A bounded
truncated Toeplitz operator need not admit a symbol in $L^\infty(\T)$ in \cite{BCFMT,Sarason}. Thus boundedness of the compression, which is initially defined by an
$L^2$ symbol, need not produce an $L^\infty$ representative. Determining
when the natural bounded-symbol map is onto is therefore a Nehari-type
problem for model spaces. Our main theorem shows that this
operator-theoretic property is governed exactly by the geometry of the
sublevel sets of the underlying inner function.

Sarason initiated the systematic operator-theoretic study of this class
and raised the bounded-symbol problem in \cite{Sarason}.  Baranov, Chalendar, Fricain, Mashreghi, and Timotin constructed
bounded truncated Toeplitz operators without bounded symbols, including
rank-one examples, and proposed the following geometric characterization
\cite[p.~2700]{BCFMT}: every operator in $\TT(\theta)$ has a bounded
symbol if and only if $\theta$ is one-component. Here $\theta$ is
\emph{one-component} if
\begin{equation}\label{intro:cls}
 \Omega_\theta(r)=\{z\in\D:\abs{\theta(z)}<r\}
\end{equation}
is nonempty and connected for some $r\in(0,1)$. This class arose in
Cohn's work on Carleson embeddings \cite{Cohn}. Aleksandrov
characterized it through independence of the embedding classes from
the exponent \cite{Aleksandrov}; Nicolau and Reijonen gave a criterion
in terms of the zeros and the singular measure \cite{NicolauReijonen}. Baranov, Bessonov, and Kapustin proved the sufficient direction
\cite[Corollary~2.5]{BBK} and restated the remaining converse as
\cite[Conjecture~2.6]{BBK}. Their Theorem~2.4 identifies the same
bounded-symbol property with an equality of Carleson embedding classes
and a weak factorization property. We prove the necessary direction for
every nonconstant inner function by showing that, if $\theta$ is not
one-component, then the quantities $\beta_\theta(R_a^\theta)$,
$a\in\D$, are unbounded.

For $A\in\TT(\theta)$, set
\begin{equation}\label{intro:beta}
 \beta_\theta(A)
 =\inf\{\norm{\varphi}_\infty:\varphi\in L^\infty(\T),\
                         A_\varphi^\theta=A\},
\end{equation}
with the infimum of the empty set understood to be $+\infty$.
We use normalized arc measure $dm(e^{it})=dt/(2\pi)$ and inner products
linear in the first argument. The reproducing kernel, its canonical conjugate, its squared norm,
and the normalized rank-one operator used below are
\begin{align}
 k_a^\theta(z)&=\frac{1-\overline{\theta(a)}\theta(z)}{1-\bar a z},
 &\widetilde k_a^\theta(z)&=\frac{\theta(z)-\theta(a)}{z-a},\notag\\
 D_\theta(a)&=\frac{1-\abs{\theta(a)}^2}{1-\abs a^2},
 &R_a^\theta&=\frac{\widetilde k_a^\theta\otimes k_a^\theta}{D_\theta(a)}.
 \label{intro:rankone}
\end{align}
The difference quotient is interpreted by analytic continuation at $z=a$.
Our tensor convention is $(u\otimes v)f=\ip{f}{v}u$.
For every $a\in\D$, the operator $R_a^\theta$ is a norm-one truncated
Toeplitz operator.
\begin{theorem}\label{thm:main}
For every nonconstant inner function $\theta$, the following conditions
are equivalent:
\begin{enumerate}
\item $\theta$ is one-component;
\item every operator in $\TT(\theta)$ has a bounded symbol;
\item $\displaystyle \sup_{a\in\D}\beta_\theta(R_a^\theta)<\infty$.
\end{enumerate}
\end{theorem}

Although every $R_a^\theta$ has a bounded symbol, condition~\textup{(iii)}
requires the corresponding symbol norms to be bounded uniformly in
$a$. Theorem \ref{thm:quantitative} shows that this uniform estimate fails whenever
$\theta$ is not one-component.

The operators $R_a^\theta$, $a\in\D$, already occur in the
necessary-condition argument of \cite[Theorem~5.6]{BCFMT}: the open
mapping theorem gives uniform bounds for representing symbols, which
imply $L^p$ estimates for reproducing kernels. Here we prove the
quantitative lower bound below from a small hyperbolic derivative by
using two inner divisors obtained from separated portions of the
canonical factorization. Together with the
geometric selection lemma, this proves the converse for every inner
function.

The quantitative estimate needed for the converse implication is expressed
in terms of the hyperbolic derivative
\begin{equation}\label{intro:q}
 q_\theta(a)=
 \frac{(1-\abs a^2)\abs{\theta'(a)}}{1-\abs{\theta(a)}^2}.
\end{equation}
By the Schwarz--Pick lemma, $0\le q_\theta(a)\le1$.

\begin{theorem}\label{thm:quantitative}
There is an absolute constant $c>0$ with the following property.
Let $\theta$ be a nonconstant inner function and $a\in\D$. If
\[
 t=-\log\abs{\theta(a)}\in(0,\tfrac12],
 \qquad q_\theta(a)\le\tfrac12,
\]
then
\begin{equation}\label{intro:quantitative}
 \beta_\theta(R_a^\theta)\ge\frac{c}{t\log(e/t)}.
\end{equation}
The constant is independent of $\theta$, $a$, the zero multiplicities,
and the singular measure in the canonical factorization.
\end{theorem}

The order in \eqref{intro:quantitative} is optimal uniformly over the
class in Theorem~\ref{thm:quantitative}. Indeed, for the degree-two
Blaschke products
\[
 \theta_r(z)=\frac{r-z^2}{1-rz^2},\qquad t=-\log r,
\]
Section~\ref{sec:examples} proves the precise asymptotic formula
\[
 \beta_{\theta_r}(R_0^{\theta_r})
 \sim\frac{\pi}{2t\log(1/t)}
 \qquad(t\downarrow0).
\]

\begin{corollary}\label{cor:divergence}
If $\theta$ is not one-component, there are points $a_n\in\D$ such that
$\abs{\theta(a_n)}\to1$, $q_\theta(a_n)\to0$, and
\[
 \norm{R_{a_n}^\theta}=1,\qquad
 \beta_\theta(R_{a_n}^\theta)\longrightarrow\infty.
\]
\end{corollary}

The same divergence also gives the following trace-class consequence.
\begin{corollary}\label{cor:traceclass}
If $\theta$ is not one-component, there is a trace-class operator
$A\in\TT(\theta)$ with no bounded symbol.
\end{corollary}

Other rank-one truncated Toeplitz operators without bounded symbols,
and hence trace-class examples, were already obtained in
\cite[Theorem~5.3]{BCFMT}.
Corollary~\ref{cor:traceclass} is a functional-analytic consequence
of Corollary~\ref{cor:divergence}; its additional content is the
conclusion for every non-one-component $\theta$.

\subsection{Consequences}\label{sec:consequences}

The following corollaries combine Theorem~\ref{thm:main} with the embedding,
duality, and symbol-representation results cited in their proofs.
In each case, the new ingredient is the necessary implication in
Theorem~\ref{thm:main}.

\subsubsection*{Embeddings and weak factorization}
For an inner function $\theta$ and $1\le p<\infty$, put
$K_\theta^p=H^p\cap \theta\overline{zH^p}$. Let $\mathcal C_p(\theta)$ be the class of
finite complex Borel measures $\mu$ on $\T$ for which
\[
 \norm f_{L^p(|\mu|)}\le C_\mu\norm f_{H^p},
 \qquad f\in K_\theta^p\cap C(\overline\D).
\]
The boundary-value convention is recalled in
Section~\ref{sec:application-identities}. Let $K_\theta\odot K_\theta$
be the set of all $h\in H^1$ admitting a representation
$h=\sum_j f_jg_j$, where $f_j,g_j\in K_\theta$ and
$\sum_j\norm{f_j}_2\norm{g_j}_2<\infty$; the series then converges in
$H^1$. Equip this space with the norm
$\norm h_\odot$ given by the infimum of these sums over all such
representations. Put
\[
 Y_\theta=H^1\cap\theta^2\overline{z^2H^1}.
\]

\begin{corollary}[Embeddings and weak factorization]
\label{cor:embedding-factorization}
The following conditions are equivalent:
\begin{enumerate}
\item $\theta$ is one-component;
\item $\mathcal C_1(\theta^2)=\mathcal C_2(\theta^2)$;
\item $K_\theta\odot K_\theta=Y_\theta$.
\end{enumerate}
When they hold, there is a constant $F_\theta<\infty$ such that
\[
 \norm h_1\le\norm h_\odot\le F_\theta\norm h_1,
 \qquad h\in Y_\theta.
\]
\end{corollary}

\begin{proof}
Combine Theorem~\ref{thm:main} with the equivalences and norm estimates
in \cite[Theorem~2.4]{BBK}.
\end{proof}

\subsubsection*{Compact symbol representations}
For $\varphi\in L^2(\T)$, we use the truncated Hankel convention
\[
 \Gamma_\varphi^\theta f=P_{\bar\theta K_\theta}(\varphi f),
 \qquad f\in K_\theta\cap H^\infty,
\]
where $P_{\bar\theta K_\theta}$ is the orthogonal projection in $L^2(\T)$
onto $\bar\theta K_\theta$. Write $\TT_c(\theta)$ for the compact
operators in $\TT(\theta)$. Using Hankel operators on Hardy space, Ma and Zheng \cite{MaZheng,MaZhengFiniteRank} characterized compact truncated Toeplitz operators with bounded symbols and finite-rank truncated Toeplitz operators.

The Toeplitz--Hankel correspondence and the norm identities used below
are stated in Lemma~\ref{pre:hankel-symbols}.

\begin{corollary}[Compact symbol representations]
\label{cor:compact-symbols}
The following conditions are equivalent:
\begin{enumerate}
\item $\theta$ is one-component;
\item each $A\in\TT_c(\theta)$ has the form
      $A=A_{\theta c}^\theta$ with $c\in C(\T)$;
\item each compact truncated Hankel operator on $K_\theta$ has the
      form $\Gamma_c^\theta$ with $c\in C(\T)$.
\end{enumerate}
\end{corollary}

\begin{proof}
The equivalence of (i) and (ii) follows from
Theorem~\ref{thm:main} and \cite[Theorem~3.2]{OLoughlin}. The equivalence
of (ii) and (iii) follows from Lemma~\ref{pre:hankel-symbols}, whose
unitary correspondence preserves compactness.
\end{proof}

The implication (i)$\Rightarrow$(ii) also follows from Bessonov's
continuous-symbol theorem \cite[Theorem~2]{BessonovCompact}, together
with Lemma~\ref{pre:hankel-symbols}. The symbols in
condition~(ii) belong to $\theta C(\T)$. By
\cite[Lemma~3.5]{BessonovCompact},
$\{A_{\theta c}^\theta:c\in C(\T)\}$ is always dense in $\TT_c(\theta)$;
for non-one-component $\theta$ it is therefore a proper dense subspace.

\subsubsection*{The Nehari-type estimate}
Set
\[
 \mathcal N_\theta=L^\infty(\T)\cap
       \bigl(H^2+\overline{\theta^2H^2}\bigr),
 \qquad
 h_a(\zeta)=\frac{1}{D_\theta(a)(\zeta-a)}\quad(\zeta\in\T).
\]
Then $h_a\in C(\T)$. Lemma~\ref{pre:hankel-symbols} identifies
$\mathcal N_\theta$ as the relevant symbol kernel and gives the
quotient-norm identities used below.

\begin{corollary}[A geometric criterion for the Nehari-type estimate]
\label{cor:nehari-criterion}
The following conditions are equivalent:
\begin{enumerate}
\item $\theta$ is one-component;
\item there is $C_\theta<\infty$ such that
\[
 \dist_{L^\infty}(\varphi,\mathcal N_\theta)
       \le C_\theta\norm{\Gamma_\varphi^\theta},
       \qquad \varphi\in L^\infty(\T);
\]
\item the estimate in~(ii) holds for all $\varphi\in C(\T)$;
\item $\displaystyle\sup_{a\in\D}
          \dist_{L^\infty}(h_a,\mathcal N_\theta)<\infty$.
\end{enumerate}
If $\theta$ is not one-component, there are $a_n\in\D$ with
$t_n=-\log\abs{\theta(a_n)}\to0$ such that
\[
 \norm{\Gamma_{h_{a_n}}^\theta}=1,
 \qquad
 \dist_{L^\infty}(h_{a_n},\mathcal N_\theta)
       \ge\frac{c}{t_n\log(e/t_n)}\longrightarrow\infty.
\]
\end{corollary}

\begin{proof}
Suppose first that (i) holds. Lemma~\ref{pre:uniform} gives
$\beta_\theta(A)\le B_\theta\norm A$ for every
$A\in\TT(\theta)$. Apply this to $A_{\theta\varphi}^\theta$ and use
Lemma~\ref{pre:hankel-symbols}; this is precisely (ii). Clearly
(ii)$\Rightarrow$(iii), and (iii)$\Rightarrow$(iv) follows by taking
$\varphi=h_a$, which is continuous. Conversely, the last identities in
Lemma~\ref{pre:hankel-symbols} identify (iv) with condition~(iii) of
Theorem~\ref{thm:main}; hence (iv)$\Rightarrow$(i). This proves the
equivalences. The quantitative assertion follows from
Corollary~\ref{cor:divergence} and Theorem~\ref{thm:quantitative}.
The implication \textup{(i)}$\Rightarrow$\textup{(ii)} also follows from \cite[Lemma~3.6]{BessonovCompact}.
\end{proof}

\subsubsection*{Clark measures and Bessonov's duality theorem}
For $\alpha\in\T$, let $\sigma_\alpha$ be the Clark measure of $\theta$ \cite{Clark},
characterized by
\[
 \Rea\frac{\alpha+\theta(z)}{\alpha-\theta(z)}
 =\int_\T\frac{1-|z|^2}{|\xi-z|^2}\,d\sigma_\alpha(\xi).
\]
For a finite positive measure $\mu$, write $a(\mu)$ for its isolated atoms and
$\rho(\mu)=\supp\mu\setminus a(\mu)$. Bessonov's Clark-measure
criterion \cite[Theorem~1]{BessonovDuality} consists of the following
conditions:
\begin{enumerate}[label=\textup{(C\arabic*)}]
\item $\mu$ is discrete, all its atoms are isolated, $m(\supp\mu)=0$,
every atom $\xi$ has adjacent atoms $\xi_-,\xi_+$ on its two sides,
and every component of $\T\setminus\rho(\mu)$ contains atoms;
\item for constants $0<A\le B<\infty$ depending on $\mu$,
$A|\xi-\xi_\pm|\le\mu\{\xi\}\le B|\xi-\xi_\pm|$ at every atom;
\item $\displaystyle
\sup_{\zeta\in a(\mu)}\left|
\int_{\T\setminus\{\zeta\}}\frac{d\mu(\xi)}{1-\bar\xi\zeta}
\right|<\infty$.
\end{enumerate}
\begin{corollary}[Clark-measure criterion]\label{cor:clark-symbols}
Suppose that $\theta$ is not a finite Blaschke product. Every operator
in $\TT(\theta)$ has a bounded symbol if and only if one Clark measure
of $\theta$ satisfies \textup{(C1)--(C3)}. Equivalently, every Clark
measure satisfies these conditions.
\end{corollary}
\begin{proof}
For each fixed $\alpha\in\T$, \cite[Theorem~1]{BessonovDuality}, applied
to $\mu=\sigma_\alpha$, identifies conditions \textup{(C1)--(C3)} with
the one-component property. The conclusion follows from
Theorem~\ref{thm:main}.
\end{proof}
We state Corollary~\ref{cor:clark-symbols} only for non-finite Blaschke products, since
the finite-dimensional case is automatic and may be treated
separately.
Moreover,
\[
zY_\theta=K_{\theta^2}^1\cap zH^1.
\]
Hence Corollary~\ref{cor:embedding-factorization} shows that the
weak-factorization property used in Bessonov's duality theory holds
exactly when $\theta$ is one-component; see
\cite[Sections~4.3--4.4]{BessonovDuality}.

\subsection{Proof outline and organization}
The proof of necessity begins with the following application of the
open mapping theorem. If
every bounded truncated Toeplitz operator has a bounded symbol, then the
open mapping theorem gives a constant $B_\theta<\infty$ such that
\[
 \beta_\theta(A)\le B_\theta\norm A,
 \qquad A\in\TT(\theta).
\]
If $\theta$ is not one-component, the geometric selection lemma produces
points $a_n$ such that
\[
 |\theta(a_n)|\longrightarrow1,
 \qquad q_\theta(a_n)\longrightarrow0.
\]
Theorem~\ref{thm:quantitative} then gives
$\beta_\theta(R_{a_n}^\theta)\to\infty$, whereas
$\norm{R_{a_n}^\theta}=1$. It therefore remains to prove the geometric selection lemma and the
quantitative lower bound in Theorem~\ref{thm:quantitative}.

By Lemma~\ref{pre:unitary}, it suffices to prove
Theorem~\ref{thm:quantitative} at $a=0$. After multiplying $\theta$ by
a unimodular constant, we may assume $\theta(0)=e^{-t}>0$ and
$q_\theta(0)\le1/2$. For zeros
$a_j=e^{-s_j}\xi_j$, $\xi_j\in\T$, and singular measure $\mu$, consider
\[
 \nu=\sum_j\sinh(s_j)\delta_{\xi_j}+\mu.
\]
Its mass is comparable to $t$, while the derivative hypothesis bounds
its first angular moment by a fixed fraction of its mass. Two separated
angular regions therefore carry masses comparable to $t$. They yield
inner factors $\phi,\psi$ of $\theta$. Writing $\theta=\phi\psi\chi$,
the factor $\chi$ contains the remaining zeros and the remaining part
of the singular measure. The functions
\[
 f=1-\phi(0)\phi,\qquad g=1-\psi(0)\psi
\]
belong to $K_\theta$ and satisfy
\[
 f(0)g(0)\ge c t^2,\qquad
 \norm{fg}_1\le C t^2\log(e/t).
\]
The lower bound in Lemma~\ref{pre:pair}, applied to $f$ and $g$, gives
\eqref{intro:quantitative}. The factor estimates use absolute
summability away from the two angular supports and a precise $L^1$
kernel identity.

The required geometric statement is a continuation lemma:
a positive lower
bound for $q_F$ on $\{\abs F>\rho\}$ forces all higher sublevel sets
of a holomorphic self-map $F$ of $\D$ to be connected. We prove this using a
continuous retraction obtained by continuing local inverse branches along
radial segments. In related
geometric terminology, one-component inner functions are characterized
by covering behavior over an outer annulus
\cite[Theorem~7.2]{IU}. The present proof does not use that
characterization or require boundary continuation of the inner function.

Section~\ref{sec:prelim} records the operator identities and
normalizations, including the open mapping argument.
Section~\ref{sec:geometry} proves the geometric
selection lemma, and Section~\ref{sec:factors} develops the factor
estimates, including the infinite-product and singular-measure limits.
Section~\ref{sec:proof} completes the proofs of the main results.
Section~\ref{sec:examples} proves the sharpness of the quantitative
order by an exact computation for a degree-two family.

\section{Preliminaries and symbol norms}\label{sec:prelim}

\subsection{Operator identities and normalization}

We retain the notation and conventions introduced in
Section~\ref{sec:intro} and identify Hardy functions with their radial
boundary functions. The domain $K_\theta\cap H^\infty$ in
\eqref{intro:tto} is dense: it contains every reproducing kernel, whose
linear span is dense in $K_\theta$. If $u$ is an inner divisor of
$\theta$, then $K_u\subset K_\theta$. For $\varphi\in L^\infty(\T)$,
\begin{equation}\label{pre:contractive-symbol-map}
  \norm{A_\varphi^\theta}\le \norm{\varphi}_\infty.
\end{equation}

By Sarason's theorem \cite[Theorem~4.2]{Sarason}, $\TT(\theta)$ is
closed in the weak operator topology, hence in the operator norm.
It is therefore a Banach space with the operator norm. In particular,
\begin{equation}\label{pre:symbol-map}
 Q_\theta:L^\infty(\T,m)\longrightarrow\TT(\theta),
 \qquad Q_\theta\varphi=A_\varphi^\theta,
\end{equation}
is a bounded linear map between Banach spaces, with $\norm{Q_\theta}\le1$.

The tensor convention in Section~\ref{sec:intro} implies
$(\alpha u)\otimes(\beta v)=\alpha\overline\beta\,(u\otimes v)$.
The canonical conjugation of $K_\theta$ is
\[
  (C_\theta f)(\zeta)=\theta(\zeta)\overline\zeta\,
                     \overline{f(\zeta)},\qquad \zeta\in\T.
\]
It is an antilinear isometric involution of $K_\theta$, and
\begin{equation}\label{pre:conjugation-identities}
  \ip{C_\theta f}{C_\theta g}=\ip{g}{f},
  \qquad \abs{C_\theta f}=\abs f\quad\text{a.e. on }\T.
\end{equation}
These facts follow, for example, from
$K_\theta=H^2\cap\theta\overline{zH^2}$.
In particular, $C_\theta$ preserves $K_\theta\cap H^\infty$.

For the kernels introduced in \eqref{intro:rankone},
\[
 \ip{f}{k_\lambda^\theta}=f(\lambda),\qquad
 C_\theta k_\lambda^\theta=\widetilde k_\lambda^\theta,
 \qquad \norm{k_\lambda^\theta}_2^2=D_\theta(\lambda).
\]
Since $\theta$ is nonconstant, $D_\theta(\lambda)>0$. Define the
unnormalized rank-one operator
\begin{equation}\label{pre:rank-one-definitions}
  T_\lambda^\theta
    =\widetilde k_\lambda^\theta\otimes k_\lambda^\theta,
  \qquad T_\lambda^\theta=D_\theta(\lambda)R_\lambda^\theta.
\end{equation}
Thus $\norm{T_\lambda^\theta}=D_\theta(\lambda)$ and
$\norm{R_\lambda^\theta}=1$.

By \eqref{intro:beta} and \eqref{pre:contractive-symbol-map},
$\beta_\theta(A)\ge\norm A$; also
$\beta_\theta(cA)=\abs c\,\beta_\theta(A)$ for $c\ne0$.
In particular, for $\lambda\in\D$,
\[
  \beta_\theta(R_\lambda^\theta)
       =\frac{\beta_\theta(T_\lambda^\theta)}{D_\theta(\lambda)}.
\]

\begin{lemma}[Uniform bounded-symbol constant]\label{pre:uniform}
The following statements are equivalent:
\begin{enumerate}
\item Every $A\in\TT(\theta)$ has a symbol in $L^\infty(\T,m)$.
\item There is a finite constant $B_\theta$ such that
\[
  \beta_\theta(A)\le B_\theta\norm A,
  \qquad A\in\TT(\theta).
\]
\end{enumerate}
\end{lemma}

\begin{proof}
The kernel of $Q_\theta$ in \eqref{pre:symbol-map} is closed, so the
quotient $Z_\theta=L^\infty/\ker Q_\theta$ is a Banach space. The
induced map
\[
 J_\theta:Z_\theta\longrightarrow\TT(\theta),
 \qquad J_\theta[\varphi]=A_\varphi^\theta,
\]
is bounded and injective. Its quotient norm satisfies
\[
 \norm{[\varphi]}_{Z_\theta}
 =\inf_{\psi\in\ker Q_\theta}\norm{\varphi+\psi}_\infty
 =\beta_\theta(A_\varphi^\theta),
\]
since two bounded symbols define the same operator precisely when
their difference belongs to $\ker Q_\theta$.

Under condition~(i), $Q_\theta$ is surjective, and hence $J_\theta$ is
a bounded bijection between Banach spaces. The open mapping theorem,
in its bounded inverse form, gives $B_\theta=\norm{J_\theta^{-1}}<\infty$.
Consequently,
\[
 \beta_\theta(A)=\norm{J_\theta^{-1}A}_{Z_\theta}
 \le B_\theta\norm A,\qquad A\in\TT(\theta).
\]

Conversely, finiteness of the infimum in
\eqref{intro:beta} implies that the set over which it is
taken is nonempty.
\end{proof}

We recall Sarason's rank-one symbol formula
\cite[Theorem~5.1(a) and its proof]{Sarason}: for every $\lambda\in\D$,
\begin{equation}\label{pre:rank-one-symbol}
  A_{\theta/(z-\lambda)}^\theta=T_\lambda^\theta.
\end{equation}
Here the symbol is its boundary function on $\T$, and is bounded.
In particular,
\begin{equation}\label{pre:origin-symbol}
  A_{\theta\bar z}^\theta=T_0^\theta.
\end{equation}
The reproducing property and the canonical conjugation give, for
all $f,g\in K_\theta$, the identity
\begin{equation}\label{pre:bilinear-pairing}
  \ip{T_\lambda^\theta f}{C_\theta g}
       =f(\lambda)g(\lambda).
\end{equation}
\begin{lemma}[A lower bound for the symbol norm]\label{pre:pair}
For every $\lambda\in\D$ and nonzero $f,g\in K_\theta$,
\begin{equation}\label{pre:symbol-norm-bound}
  \beta_\theta(R_\lambda^\theta)
  \ge
  \frac{\abs{f(\lambda)g(\lambda)}}
       {D_\theta(\lambda)\displaystyle\int_\T\abs{fg}\,dm}.
\end{equation}
\end{lemma}

\begin{proof}
Let $\varphi\in L^\infty(\T)$ be any symbol for $R_\lambda^\theta$.
By \eqref{pre:bilinear-pairing} and
$\abs{C_\theta g}=\abs g$ almost everywhere,
\[
 \frac{\abs{f(\lambda)g(\lambda)}}{D_\theta(\lambda)}
 =\abs{\ip{A_\varphi^\theta f}{C_\theta g}}
 =\left|\int_\T\varphi f\,\overline{C_\theta g}\,dm\right|
 \le\norm{\varphi}_\infty\int_\T\abs{fg}\,dm.
\]
The integral is finite and positive for nonzero $f,g\in K_\theta$:
finiteness follows from Cauchy--Schwarz, while positivity follows from
the Hardy-space uniqueness theorem, since the nonzero analytic function
$fg$ cannot have boundary values equal to zero on a set of positive
measure.
Dividing by it and taking the infimum over all such $\varphi$ proves
\eqref{pre:symbol-norm-bound}.
\end{proof}

\begin{lemma}[Automorphisms and exact preservation of symbol norms]
\label{pre:unitary}
Fix $a\in\D$, and put
\[
  s_a=\sqrt{1-\abs a^2},\qquad
  \tau_a(z)=\frac{a-z}{1-\overline a z},\qquad
  w_a(z)=\frac{s_a}{1-\overline a z},\qquad
  \eta=\theta\circ\tau_a.
\]
The weighted composition operator
\[
  (U_a h)(z)=w_a(z)h(\tau_a(z))
\]
is a self-adjoint unitary on $L^2(\T,m)$ and maps $K_\theta$ onto
$K_\eta$. For $\varphi\in L^\infty(\T,m)$,
\begin{equation}\label{pre:symbol-conjugation}
  U_aA_\varphi^\theta U_a^*=A_{\varphi\circ\tau_a}^\eta,
  \qquad
  \norm{\varphi\circ\tau_a}_\infty=\norm\varphi_\infty.
\end{equation}
More generally, conjugation by $U_a$ maps $\TT(\theta)$ onto
$\TT(\eta)$, and
\begin{equation}\label{pre:symbol-norm-invariance}
  \beta_\eta(U_aAU_a^*)=\beta_\theta(A),
  \qquad A\in\TT(\theta).
\end{equation}
For the kernels based at $a$,
\begin{equation}\label{pre:kernel-recentering}
  U_a k_a^\theta=s_a^{-1}k_0^\eta,
  \qquad
  U_a\widetilde k_a^\theta=-s_a^{-1}\widetilde k_0^\eta.
\end{equation}
Consequently,
\begin{equation}\label{pre:normalized-recentering}
  U_aR_a^\theta U_a^*=-R_0^\eta,
  \qquad
  \beta_\theta(R_a^\theta)=\beta_\eta(R_0^\eta).
\end{equation}
The hyperbolic derivative of $\eta$ at the origin is
\begin{equation}\label{pre:q-recentering}
 q_\eta(0)=q_\theta(a).
\end{equation}
\end{lemma}

\begin{proof}
The identities
\[
  \tau_a\circ\tau_a=\operatorname{id},\qquad
  w_a(z)w_a(\tau_a(z))=1,
\]
together with
\[
  \abs{\tau_a'(\zeta)}
    =\frac{1-\abs a^2}{\abs{1-\overline a\zeta}^2}
    =\abs{w_a(\zeta)}^2,
    \qquad \zeta\in\T,
\]
show by change of variables that $U_a$ is unitary and $U_a^2=I$.
Thus $U_a^*=U_a$. It maps $H^2$ onto itself and
$\theta H^2$ onto $\eta H^2$, and therefore maps $K_\theta$ onto
$K_\eta$. The assertion about $H^2$ can be checked first on analytic
polynomials: their weighted compositions are bounded analytic functions;
unitarity and density then give invariance of $H^2$, and $U_a^2=I$
gives surjectivity. In particular,
\[
  U_aP_\theta U_a^*=P_\eta,
  \qquad
  U_aM_\varphi U_a^*=M_{\varphi\circ\tau_a}.
\]
This proves the operator identity in
\eqref{pre:symbol-conjugation}. The boundary automorphism and its
inverse preserve null sets, so the essential supremum is unchanged.

For completeness, the same conjugation identity is valid for an
$L^2$ symbol on the defining dense domains. For fixed $a$, the change
of variables has a bounded Jacobian and bounded inverse Jacobian;
hence $\varphi\circ\tau_a\in L^2$ whenever $\varphi\in L^2$.
The analytic weights involved in $U_a$ and $U_a^{-1}$ are bounded,
so $U_a$ maps $K_\theta\cap H^\infty$ onto
$K_\eta\cap H^\infty$. The projection and multiplication identities
therefore apply on these domains. Passing to bounded extensions
shows that conjugation maps $\TT(\theta)$ onto
$\TT(\eta)$. Applying \eqref{pre:symbol-conjugation} in both
directions gives \eqref{pre:symbol-norm-invariance}, including when the
symbol norm is infinite.

The elementary identity
\[
  1-\overline a\tau_a(z)
    =\frac{s_a^2}{1-\overline a z}
\]
gives the first formula in \eqref{pre:kernel-recentering}. Similarly,
\[
  \tau_a(z)-a
    =-\frac{s_a^2z}{1-\overline a z}
\]
gives the second formula directly from the divided difference
definition. Thus
\[
  U_aT_a^\theta U_a^*=-s_a^{-2}T_0^\eta,
  \qquad
  D_\theta(a)=s_a^{-2}D_\eta(0).
\]
Dividing proves the normalized identity, whose sign is exactly
$-1$. Equation \eqref{pre:symbol-norm-invariance} and homogeneity then prove
the equality of the infimal bounded-symbol norms. Finally,
$\tau_a'(0)=-(1-\abs a^2)$ and $\eta(0)=\theta(a)$ give
\[
 q_\eta(0)=\frac{\abs{\eta'(0)}}{1-\abs{\eta(0)}^2}
 =\frac{(1-\abs a^2)\abs{\theta'(a)}}{1-\abs{\theta(a)}^2}
 =q_\theta(a),
\]
which proves \eqref{pre:q-recentering}.
\end{proof}

\begin{remark}[Unimodular rotations]\label{pre:rotation}
If $\eta^\sharp=c\eta$ with $\abs c=1$, then $K_{\eta^\sharp}=K_\eta$ and
$A_\varphi^{\eta^\sharp}=A_\varphi^\eta$ for every symbol $\varphi$.
Moreover,
\[
  C_{\eta^\sharp}=cC_\eta,\qquad
  k_\lambda^{\eta^\sharp}=k_\lambda^\eta,\qquad
  \widetilde k_\lambda^{\eta^\sharp}=c\widetilde k_\lambda^\eta,
  \qquad R_\lambda^{\eta^\sharp}=cR_\lambda^\eta.
\]
Hence the infimal bounded-symbol norms are unchanged. After applying
Lemma~\ref{pre:unitary} and replacing $\eta$ by
$\eta^\sharp=c\eta$, the precise
operator identity becomes
\[
  U_aR_a^\theta U_a^*=-\overline c\,R_0^{\eta^\sharp}.
\]
Thus one may arrange $\eta^\sharp(0)>0$ whenever $\eta(0)\ne0$, without
altering any operator norm or infimal bounded-symbol norm.
\end{remark}

\subsection{Boundary conventions and the Toeplitz--Hankel correspondence}
\label{sec:application-identities}
We record the conventions and calculations used in
Section~\ref{sec:consequences}. If $\mu\in\mathcal C_p(\theta)$, the
inclusion $K_\theta^p\cap C(\overline\D)\to L^p(|\mu|)$ extends continuously
to $K_\theta^p$. By Aleksandrov's theorem, recalled in
\cite[Theorem~1.1]{BBK}, this extension agrees $|\mu|$-almost everywhere
with angular boundary values. It is therefore independent of the
choice of an $m$-almost everywhere representative, even when $\mu$ is
singular with respect to $m$.
For the weak product, the Cauchy--Schwarz inequality gives convergence
in $H^1$, and
\[
 fg=\theta^2\bar z^{\,2}\overline{(C_\theta f)(C_\theta g)}
 \qquad(f,g\in K_\theta)
\]
shows that $K_\theta\odot K_\theta\subset Y_\theta$.

Several related notions of truncated Hankel operator occur in the
literature. Gu and Ma \cite{GuMa2024} characterized boundedness, finite
rank, and compactness for compressions of Hardy-space Hankel operators
to model spaces, with the standard linear flip.
Ma, Yan and Zheng \cite{MaYanZhengBigHankel} treated vanishing,
finite rank and compactness of the big truncated Hankel operator
$H_\varphi^\theta=(I-P_\theta)M_\varphi|_{K_\theta}$,
with range in $L^2(\T)\ominus K_\theta$, for $\varphi\in L^\infty(\T)$.

Our convention for $\Gamma_\varphi^\theta$ follows Bessonov. Since
$\bar\theta K_\theta=\overline{zK_\theta}\subset L^2\ominus H^2$,
the definitions give
$\Gamma_\varphi^\theta=P_{\bar\theta K_\theta}H_\varphi^\theta$.

\begin{lemma}[Hankel symbols and quotient norms]\label{pre:hankel-symbols}
Multiplication by $\bar\theta$ defines a unitary map from $K_\theta$
onto $\bar\theta K_\theta$, and
\begin{equation}\label{conseq:hankel-toeplitz}
 \Gamma_\varphi^\theta=M_{\bar\theta}A_{\theta\varphi}^\theta.
\end{equation}
The kernel of $\varphi\mapsto\Gamma_\varphi^\theta$ on $L^\infty$ is
$\mathcal N_\theta$. For every $\varphi\in L^\infty$,
\[
 \dist_{L^\infty}(\varphi,\mathcal N_\theta)
 =\beta_\theta(A_{\theta\varphi}^\theta),\qquad
 \norm{\Gamma_\varphi^\theta}=\norm{A_{\theta\varphi}^\theta}.
\]
For the functions $h_a$ defined in Section~\ref{sec:consequences},
\[
 A_{\theta h_a}^\theta=R_a^\theta,\qquad
 \norm{\Gamma_{h_a}^\theta}=1,\qquad
 \dist_{L^\infty}(h_a,\mathcal N_\theta)=\beta_\theta(R_a^\theta).
\]
\end{lemma}

\begin{proof}
The identity
$P_{\bar\theta K_\theta}=M_{\bar\theta}P_\theta M_\theta$
gives \eqref{conseq:hankel-toeplitz} on $K_\theta\cap H^\infty$,
and hence for the bounded extensions. This is the correspondence
in \cite[Lemma~3.3]{BessonovCompact}; left multiplication by this
unitary preserves norms, compactness and singular values.
The kernel description is \cite[Lemma~3.4]{BessonovCompact}, and
in particular $\mathcal N_\theta$ is norm closed.

Multiplication of boundary symbols by $\theta$ is an isometric
bijection of $L^\infty$. By the correspondence, a bounded symbol
$\theta\psi$ represents $A_{\theta\varphi}^\theta$ precisely when
$\psi-\varphi\in\mathcal N_\theta$. Taking the infimum over these
symbols proves the distance identity. Finally,
\eqref{pre:rank-one-symbol} gives
$A_{\theta h_a}^\theta=R_a^\theta$, and $\norm{R_a^\theta}=1$
gives the remaining assertions.
\end{proof}

\section{A geometric consequence of a lower bound for the hyperbolic derivative}
\label{sec:geometry}

For a holomorphic map $F:\D\to\D$, set
\[
 q_F(z)=\frac{(1-\abs{z}^2)\abs{F'(z)}}{1-\abs{F(z)}^2},
 \qquad z\in\D.
\]
We use the hyperbolic length element
$\abs{dz}/(1-\abs{z}^2)$ on $\D$.  The geometric argument below does
not require $F$ to be inner.

Corollary~\ref{geom:sequence} shows that every non-one-component inner
function admits a sequence satisfying the hypotheses of
Theorem~\ref{thm:quantitative}, with $t$ tending to zero. This is the
only conclusion of the section used in the proof of the main theorem;
the proof of Theorem~\ref{thm:quantitative} is independent of this
section.

\begin{lemma}\label{geom:retract}
Let $F:\D\to\D$ be holomorphic.  Suppose that there are constants
$0<\rho<1$ and $\delta>0$ such that
\begin{equation}\label{geom:lower-bound}
 q_F(z)\geq\delta
 \qquad\text{whenever }\abs{F(z)}>\rho.
\end{equation}
For each $R\in(\rho,1)$, the set
\[
 S_R=\{z\in\D:\abs{F(z)}\leq R\}
\]
is a continuous retract of $\D$.  In particular, $S_R$ is nonempty
and path connected.  Consequently, for every $r\in(\rho,1)$, the
open sublevel set
\[
 \Omega_r=\{z\in\D:\abs{F(z)}<r\}
\]
is nonempty and path connected.
\end{lemma}

\begin{proof}
Fix $R\in(\rho,1)$, and let $d_h$ be the distance induced by
$|dz|/(1-|z|^2)$. We use
$d_h(0,w)=\operatorname{arctanh}|w|$ and $|v-w|\le d_h(v,w)$.
If $F$ is constant, the hypothesis forces $|F|\le\rho$, and all
the asserted sublevel sets equal $\D$. Otherwise, the assumed
lower bound implies that $F'$ has no zeros in $\{\abs{F}>\rho\}$.
For $z\in\D$ with
$a=\abs{F(z)}>R$, put $u=F(z)/a$ and consider the radial segment
\begin{equation}\label{geom:radial-path}
 \gamma_z(t)=\bigl(a-(a-R)t\bigr)u,\qquad 0\leq t\leq1.
\end{equation}
Since $F'(z)\ne0$, the inverse function theorem gives a lift along
an initial part of this segment. Two partial lifts starting at $z$
agree on their common interval: the set of times at which they
agree is nonempty, closed by continuity, and open by local
injectivity of $F$. Let $T\in(0,1]$ be the supremum of the numbers
$b\in(0,1]$ for which a lift exists on $[0,b]$.
The partial lifts are compatible by uniqueness and therefore define a
lift $\eta_z$ on $[0,T)$, with
\[
 F(\eta_z(t))=\gamma_z(t),\qquad \eta_z(0)=z.
\]
As long as the continuation exists, differentiation gives
\begin{align}
 \frac{\abs{\eta_z'(t)}}{1-\abs{\eta_z(t)}^2}
 &=\frac{\abs{\gamma_z'(t)}}
 {q_F(\eta_z(t))(1-\abs{\gamma_z(t)}^2)}\notag\\
 &\leq\frac1\delta
 \frac{\abs{\gamma_z'(t)}}{1-\abs{\gamma_z(t)}^2}.
 \label{geom:speed}
\end{align}
Here $\abs{F(\eta_z(t))}\geq R>\rho$, so the lower bound for
$q_F$ applies.  The right-hand side is integrable on $[0,1]$, and
\begin{equation}\label{geom:length}
 \frac1\delta\int_0^1
 \frac{\abs{\gamma_z'(t)}}{1-\abs{\gamma_z(t)}^2}\,dt
 =\frac1\delta\bigl(\operatorname{arctanh}a
                    -\operatorname{arctanh}R\bigr)<\infty.
\end{equation}
Denote the finite bound in \eqref{geom:length} by $\Lambda_z$.
For $0\le t<T$,
\[
 \operatorname{arctanh}|\eta_z(t)|
 =d_h(0,\eta_z(t))
 \le\operatorname{arctanh}|z|+\Lambda_z.
\]
Thus the lift stays in a fixed compact subdisk of $\D$.
For $0\le t_1<t_2<T$, \eqref{geom:speed} bounds
$d_h(\eta_z(t_1),\eta_z(t_2))$ by the corresponding integral over
$[t_1,t_2]$. Since the integrand is integrable on $[0,1]$, the
integrals over $[t_1,t_2]$ tend to zero as $t_1,t_2\uparrow T$.
Thus the lift is Cauchy in the Euclidean metric and has a limit
$\eta_*\in\D$. By continuity,
\[
 F(\eta_*)=\gamma_z(T),\qquad
 \abs{F(\eta_*)}\geq R>\rho.
\]
Consequently $F'(\eta_*)\ne0$. Choose neighborhoods $U$ of $\eta_*$
and $V$ of $\gamma_z(T)$ such that $F:U\to V$ is biholomorphic,
and write $g:V\to U$ for its inverse. For $t<T$ sufficiently close
to $T$, both $\eta_z(t)\in U$ and $\gamma_z(t)\in V$, so
$\eta_z(t)=g(\gamma_z(t))$. If $T<1$, this formula continues the
lift beyond $T$, contradicting its definition. Hence $T=1$, and
setting $\eta_z(1)=\eta_*$ completes the unique lift to $[0,1]$.

Define
\begin{equation}\label{geom:retraction-map}
 P_R(z)=
 \begin{cases}
 z,&\abs{F(z)}\leq R,\\
 \eta_z(1),&\abs{F(z)}>R.
 \end{cases}
\end{equation}
This map is defined on all of $\D$, takes values in $S_R$, and
fixes every point of $S_R$.

To prove continuity at a point $z_0$ with $\abs{F(z_0)}>R$, put
$\eta_0=\eta_{z_0}$ and $\gamma_0=\gamma_{z_0}$. At each
$s\in[0,1]$, choose inverse neighborhoods $U_s$ of $\eta_0(s)$
and $V_s$ of $\gamma_0(s)$, and let $g_s:V_s\to U_s$ be the
inverse of $F$. By continuity of the two paths, there is a
relative open interval $I_s$ about $s$ on which
$\eta_0(t)\in U_s$, $\gamma_0(t)\in V_s$, and
$\eta_0(t)=g_s(\gamma_0(t))$. A finite subcover and a Lebesgue
number for this cover of $[0,1]$ give a subdivision
\[
 0=t_0<t_1<\cdots<t_N=1
\]
such that each $[t_{j-1},t_j]$ lies in one of the intervals $I_s$.
Relabel the corresponding inverse branches as $g_1,\ldots,g_N$.
Their open domains $V_j$ contain the compact path pieces, and
\[
 \eta_0(t)=g_j(\gamma_0(t))
 \qquad(t_{j-1}\leq t\leq t_j).
\]
At a joining point $w_j=\gamma_0(t_j)$, both adjacent branches
take the value $\eta_0(t_j)$. Choose a neighborhood of this value
on which $F$ is injective. By continuity, a sufficiently small
disk $D_j$ about $w_j$, contained in $V_j\cap V_{j+1}$, has both
branch images in that neighborhood. The identities
$F\circ g_j=F\circ g_{j+1}=\operatorname{id}$ then imply
$g_j=g_{j+1}$ on $D_j$.

For $z$ near $z_0$ with $|F(z)|>R$, the formula
\[
 \gamma_z(t)=(1-t)F(z)+tR\frac{F(z)}{|F(z)|}
\]
shows that $\gamma_z\to\gamma_0$ uniformly on $[0,1]$ as $z\to z_0$.
Each compact piece $\gamma_0([t_{j-1},t_j])$ has positive distance
from the complement of $V_j$. Since there are only finitely many
pieces and joining disks, one neighborhood of $z_0$ ensures
$\gamma_z([t_{j-1},t_j])\subset V_j$ and
$\gamma_z(t_j)\in D_j$ for every relevant $j$. Shrink this
neighborhood also into the open set $g_1(V_1)$ containing $z_0$.
Then $g_1(F(z))=z$, and the formulas $g_j\circ\gamma_z$ match at
the joins and give a lift starting at $z$. By uniqueness this is
$\eta_z$. In particular,
\[
 P_R(z)=g_N\left(R\frac{F(z)}{\abs{F(z)}}\right)
\]
near $z_0$, which proves continuity there.

At a point $z_*$ with $|F(z_*)|=R$, the length estimate gives a
direct continuity argument. Write $x_+=\max\{x,0\}$. For every
$z\in\D$, the definition of $P_R$ and \eqref{geom:length} imply
\[
 |P_R(z)-z|
 \le d_h(P_R(z),z)
 \le\frac1\delta
 \bigl(\operatorname{arctanh}|F(z)|
             -\operatorname{arctanh}R\bigr)_+,
\]
and hence
\[
 |P_R(z)-z_*|
 \le |z-z_*|+\frac1\delta
 \bigl(\operatorname{arctanh}|F(z)|
             -\operatorname{arctanh}R\bigr)_+.
\]
The right-hand side tends to zero as $z\to z_*$. Finally, where
$|F|<R$, continuity of $F$ makes $P_R$ locally the identity.
Hence $P_R$ is a continuous retraction.

Since $P_R$ takes values in $S_R$ and fixes its points,
$P_R(\D)=S_R$. This image is nonempty and path connected because
$\D$ has these properties. For $\rho<r<1$, choosing any $R\in(\rho,r)$
shows that $\Omega_r$ is nonempty.  Given $z_1,z_2\in\Omega_r$,
choose
\[
 \max\{\rho,\abs{F(z_1)},\abs{F(z_2)}\}<R<r.
\]
Composing any path in $\D$ from $z_1$ to $z_2$ with $P_R$ gives a
path with the same endpoints contained in $S_R\subset\Omega_r$.
Thus $\Omega_r$ is path connected.
\end{proof}

\begin{corollary}\label{geom:sequence}
Let $\theta$ be a nonconstant inner function which is not
one-component.  Then, for every $\rho\in(0,1)$,
\begin{equation}\label{geom:infimum}
 \inf_{\{z\in\D:\abs{\theta(z)}>\rho\}}q_\theta(z)=0.
\end{equation}
In particular, there are points $a_n\in\D$, $n\geq2$, such that
\begin{equation}\label{geom:chosen-points}
 \abs{\theta(a_n)}>1-\frac1n,
 \qquad q_\theta(a_n)<\frac1n.
\end{equation}
Thus $\abs{\theta(a_n)}\to1$ and $q_\theta(a_n)\to0$.
\end{corollary}

\begin{proof}
Since $\theta$ is inner, $\sup_{\D}\abs{\theta}=1$; hence the set
over which the infimum in \eqref{geom:infimum} is taken is
nonempty.  If that infimum were positive for some $\rho$,
Lemma~\ref{geom:retract} would make every sublevel set
$\{\abs{\theta}<r\}$ with $\rho<r<1$ nonempty and connected.
This would make $\theta$ one-component, a contradiction.
Applying \eqref{geom:infimum} with $\rho=1-1/n$ gives
\eqref{geom:chosen-points}.
\end{proof}

\section{Separated canonical factors and product test functions}
\label{sec:factors}

Throughout this section, $m$ denotes normalized Lebesgue measure on $\T$.
For an inner function $\vartheta$, we write
$K_\vartheta=H^2\ominus\vartheta H^2$ and
\[
 k_0^\vartheta=1-\overline{\vartheta(0)}\vartheta.
\]
All constants described as universal are independent of the inner function
and of the parameter $t$.  We shall use only $0<t\le t_0:=1/2$.
The zero index set $\mathcal I$ may be empty, finite, or $\mathbb N$;
in the finite case we use $\mathcal I=\{1,\ldots,N_*\}$.
All zero sums and products are taken over this set, with multiplicities.
Empty sums equal $0$ and empty products equal $1$; finite partial
products are kept constant after the last factor. These conventions
also apply to every selected subfactor below. The singular measure
may be nonzero even when $\mathcal I$ is finite or empty.
For clarity, the hyperbolic derivative ratio at the origin is
\[
 q_\theta(0)=\frac{\abs{\theta'(0)}}{1-\abs{\theta(0)}^2}.
\]

\begin{lemma}[Canonical mass and its first moment]
\label{factor:canonical}
Suppose that $\theta$ is inner and $\theta(0)=e^{-t}>0$, where
$0<t\le 1/2$.  List its zeros with multiplicity as
$a_j=\rho_j\eta_j$, where $0<\rho_j<1$ and $\eta_j\in\T$, and put
$s_j=-\log\rho_j$.  Let $\mu$ be its canonical singular measure.
The positive measures
\[
 \sigma=\sum_j s_j\delta_{\eta_j}+\mu,
 \qquad
 \nu=\sum_j\sinh(s_j)\delta_{\eta_j}+\mu
\]
satisfy
\begin{equation}
 \begin{gathered}
  \sigma(\T)=t,\qquad C_0=2\sinh(1/2),\\
  \sigma(G)\le\nu(G)\le C_0\sigma(G)
  \quad\text{for every Borel set }G\subset\T.
 \end{gathered}
 \label{factor:measure-comparison}
\end{equation}
Moreover,
\begin{equation}
 \frac{\theta'(0)}{\theta(0)}
 =-2\int_\T\overline\eta\,d\nu(\eta).
 \label{factor:moment-identity}
\end{equation}
If $q_\theta(0)\le 1/2$, then, with $M=\nu(\T)$,
\begin{equation}
 \abs{\int_\T\eta\,d\nu(\eta)}\le\frac35 M.
 \label{factor:moment-bound}
\end{equation}
\end{lemma}

\begin{proof}
There is no zero at the origin.  With all Blaschke factors normalized
to be positive there, the canonical factorization is
\begin{equation}
 \theta(z)=
 \prod_j\frac{\rho_j-\overline\eta_j z}
                  {1-\rho_j\overline\eta_j z}
 \exp\left(-\int_\T\frac{\eta+z}{\eta-z}\,d\mu(\eta)\right).
 \label{factor:canonical-product}
\end{equation}
The unimodular constant equals $1$ because $\theta(0)>0$.
Evaluation at the origin gives
\[
 \sum_j s_j+\mu(\T)=t.
\]
In particular $s_j\le t\le1/2$, and
\[
 s\le\sinh s\le C_0s\qquad(0\le s\le1/2).
\]
This proves \eqref{factor:measure-comparison}, including its assertion
for every Borel set and for zeros counted with multiplicity.
Write $B_j(z)=(\rho_j-\overline\eta_jz)/(1-\rho_j\overline\eta_jz)$.
Direct differentiation gives
\[
 \frac{B_j'(0)}{B_j(0)}
 =-(\rho_j^{-1}-\rho_j)\overline\eta_j
 =-2\sinh(s_j)\overline\eta_j.
\]
Let $B_N=\prod_{j\in\mathcal I,\,j\le N}B_j$ and let $B$ be the
canonical Blaschke product. If $\mathcal I$ is finite or empty,
$B_N=B$ for all sufficiently large $N$. In general, local uniform
convergence and Cauchy's formula give, for
any fixed $0<r<1$,
\[
 |B_N'(0)-B'(0)|\le r^{-1}\sup_{|z|=r}|B_N(z)-B(z)|\longrightarrow0.
\]
Also $B_N(0)\to e^{-\sum_j s_j}>0$. Passing to the limit in the
finite-product logarithmic-derivative identities gives
$B'(0)/B(0)=\sum_{j\in\mathcal I}B_j'(0)/B_j(0)$.
The resulting sum is absolutely convergent because
$\sum_j\sinh(s_j)\le C_0t$.
For $L(z)=\int_\T(\eta+z)/(\eta-z)\,d\mu(\eta)$,
the difference quotient is
\[
 \frac{L(h)-L(0)}{h}=\int_\T\frac{2}{\eta-h}\,d\mu(\eta),
 \qquad 0<|h|\le r<1.
\]
Its integrand converges pointwise to $2\overline\eta$ as $h\to0$
and has the $\mu$-integrable majorant $2/(1-r)$. Dominated convergence
(on real and imaginary parts) proves
$L'(0)=2\int_\T\overline\eta\,d\mu(\eta)$.
Combining the two contributions proves \eqref{factor:moment-identity}.

If $q_\theta(0)\le1/2$, then
\[
 \abs{\int_\T\eta\,d\nu(\eta)}
 =\frac{\abs{\theta'(0)}}{2\theta(0)}
 \le\frac{1-e^{-2t}}{4e^{-t}}
 =\frac12\sinh t.
\]
Since $M\ge t$ and $s\mapsto\sinh(s)/s$ is increasing on $(0,\infty)$,
\[
 \frac1M\abs{\int_\T\eta\,d\nu(\eta)}
 \le\frac{\sinh t}{2t}
 \le\sinh(1/2)<\frac35.
\]
\end{proof}

\begin{lemma}[Two separated sets carrying positive mass]
\label{factor:split}
Let $\nu$ be a finite positive measure on $\T$, with
$M=\nu(\T)>0$, and suppose that
\[
 \abs{\int_\T\eta\,d\nu(\eta)}\le\frac35 M.
\]
There are disjoint Borel sets $E,F\subset\T$ such that
\begin{equation}
 \nu(E)\ge\frac{M}{16},\qquad
 \nu(F)\ge\gamma M,
 \qquad
 \gamma=\frac{\sqrt3/2-3/5}{1+\sqrt3/2}>0,
 \label{factor:split-masses}
\end{equation}
and
\begin{equation}
 \operatorname{dist}(\overline E,\overline F)\ge\Delta_0,
 \qquad \Delta_0=2\sin\frac{5\pi}{96}>0.
 \label{factor:split-distance}
\end{equation}
Here the distance in \eqref{factor:split-distance} is Euclidean distance.
One may take $E$ to be a half-open arc of angular length $\pi/8$ and
$F$ to be the complement of an open arc concentric with $E$.
\end{lemma}

\begin{proof}
Partition $\T$ into $16$ half-open arcs of equal angular length,
assigning every endpoint to exactly one arc.  Choose an arc $E$ of
maximal $\nu$-mass, so that $\nu(E)\ge M/16$, and denote its center
by $\xi\in\T$.  Let $J$ be the open arc with center $\xi$ and angular
half-width $\pi/6$, and set $F=\T\setminus J$.
The closure of $E$ has angular half-width $\pi/16$, so $E\subset J$,
and the circular angular distance between $\overline E$ and $F$ is
at least
\[
 \frac\pi6-\frac\pi{16}=\frac{5\pi}{48}.
\]
This proves \eqref{factor:split-distance}. For $\eta\in J$, write
$\eta=\xi e^{iu}$ with $|u|<\pi/6$. Then
\[
 \operatorname{Re}(\overline\xi\eta)=\cos u
 \ge\cos(\pi/6)=\frac{\sqrt3}{2}.
\]
On all of $\T$ we also have
$\operatorname{Re}(\overline\xi\eta)\ge-1$. Therefore
\[
 \frac35 M
 \ge\operatorname{Re}\left(\overline\xi
                   \int_\T\eta\,d\nu(\eta)\right)
 \ge\frac{\sqrt3}{2}\nu(J)-\nu(F)
 =\frac{\sqrt3}{2}M-
        \left(1+\frac{\sqrt3}{2}\right)\nu(F).
\]
Rearranging gives the estimate for $\nu(F)$.
Atoms at partition endpoints or at the endpoints of $J$ cause no
ambiguity: the former are assigned by the half-open partition, and
the latter belong to $F$.
\end{proof}

With the notation and canonical factorization of
Lemma~\ref{factor:canonical}, for a Borel
set $G\subset\T$ define the canonical subfactor
\begin{equation}
 \theta_G(z)=
 \prod_{\eta_j\in G}
       \frac{\rho_j-\overline\eta_jz}
            {1-\rho_j\overline\eta_jz}
 \exp\left(-\int_G\frac{\eta+z}{\eta-z}\,d\mu(\eta)\right).
 \label{factor:subfactor}
\end{equation}
Zeros are selected with their multiplicities.
The restricted singular measure is finite and singular, and every
selected zero subsequence satisfies the Blaschke condition. Thus
$\theta_G$ is inner (with $\theta_G\equiv1$ if $\sigma(G)=0$).
More explicitly, for $|z|\le r<1$,
\[
 |1-B_j(z)|\le\frac{1+r}{1-r}(1-\rho_j)
              \le\frac{1+r}{1-r}s_j.
\]
The product therefore converges locally uniformly in $\D$, and
\begin{equation}
 \theta_G(0)=e^{-\sigma(G)}.
 \label{factor:subfactor-zero}
\end{equation}
For every finite Borel partition of $\T$, the corresponding subfactors
multiply to $\theta$. Indeed, the displayed summability permits
regrouping the Blaschke factors on compact subsets, while additivity
of the restricted singular measures gives the exponential factor.
To justify the regrouping explicitly, on $|z|\le r$ the change
caused by adjoining or deleting any finite collection of tail factors
is at most $(1+r)/(1-r)$ times the sum of their $s_j$'s. This tends
to zero uniformly once a sufficiently large finite set of zero
occurrences has been retained, independently of their order.

\begin{lemma}[Uniform estimate away from the selected boundary set]
\label{factor:off}
Let $G\subset\T$ be Borel, and let $A\subset\T$ be a closed set
containing $G$.  For $0<\delta\le2$, put
\[
 X(A,\delta)=
 \{u\zeta:0\le u\le1,\ \zeta\in\T,
                 \operatorname{dist}(\zeta,A)\ge\delta\}.
\]
The subfactor $\theta_G$ extends continuously to
$\D\cup X(A,\delta)$, and
\begin{equation}
 \abs{1-\theta_G(u\zeta)}
 \le\frac4\delta\sigma(G)
 \qquad(u\zeta\in X(A,\delta)).
 \label{factor:off-estimate}
\end{equation}
In particular, its values at $u=1$ in this formula are its radial
boundary limits.  No assumption on atomicity or on the original
support of the singular measure is required.
\end{lemma}

\begin{proof}
We use the convention $\operatorname{dist}(\zeta,\varnothing)=+\infty$.
If $A=\varnothing$, then $G=\varnothing$ and $\theta_G=1$, so the
assertion is immediate. If $X(A,\delta)=\varnothing$, there is nothing
to prove. Otherwise $A$ and $X(A,\delta)$ are nonempty compact sets:
the latter is the continuous image of $[0,1]$ times the closed set
$\{\zeta\in\T:\operatorname{dist}(\zeta,A)\ge\delta\}$.

For $\omega\in\T$ and $0\le v\le1$,
\begin{equation}
 \abs{1-v\omega}\ge\frac12\abs{1-\omega}.
 \label{factor:denominator}
\end{equation}
Indeed, $1-v\le\abs{1-v\omega}$, and the triangle inequality gives
$\abs{1-\omega}\le\abs{1-v\omega}+(1-v)$.
For a selected zero and $z=u\zeta\in X(A,\delta)$, write
\[
 B_j(z)=\frac{\rho_j-\overline\eta_jz}
                  {1-\rho_j\overline\eta_jz}.
\]
Using \eqref{factor:denominator} with
$v=\rho_ju$ and $\omega=\overline\eta_j\zeta$, we obtain
\[
 \abs{1-B_j(u\zeta)}
 =\frac{(1-\rho_j)\abs{1+u\overline\eta_j\zeta}}
        {\abs{1-\rho_ju\overline\eta_j\zeta}}
 \le\frac4\delta(1-\rho_j)
 \le\frac4\delta s_j.
\]
Every finite product has modulus at most $1$ on the closed disk, since
\[
 |1-\rho_j\overline\eta_jz|^2-|\rho_j-\overline\eta_jz|^2
 =(1-\rho_j^2)(1-|z|^2)\ge0.
\]
Its telescoping expansion therefore gives
\[
 \abs{1-\prod_{j\in\mathcal F}B_j(u\zeta)}
 \le\frac4\delta\sum_{j\in\mathcal F}s_j
\]
for every finite set $\mathcal F$ of selected zero occurrences.
If the selected list is finite or empty, let $P_G$ denote its full
product; the estimate and continuity already follow. For an infinite selected
list, let $P_n$ be the product of its first $n$ factors. Then
\[
 \sup_{z\in X(A,\delta)}|P_m(z)-P_n(z)|
 \le\frac4\delta\sum_{j=n+1}^{m}s_j\longrightarrow0
 \qquad(m>n\to\infty),
\]
where the selected zero occurrences have been relabelled.
This proves uniform convergence on $X(A,\delta)$.
Each finite product is continuous on the closed disk, so the uniform
limit $P_G$ is continuous on $X(A,\delta)$ and satisfies $|P_G|\le1$.
On $X(A,\delta)\cap\D$ it agrees with the canonical Blaschke product,
since the same partial products converge locally uniformly in $\D$.
Passing to the limit in the finite-product estimate gives
$|1-P_G(z)|\le(4/\delta)\sum_{\eta_j\in G}s_j$.

For the singular part, write
\[
 L_G(z)=\int_G\frac{\eta+z}{\eta-z}\,d\mu(\eta).
\]
Equation \eqref{factor:denominator}, now with $v=u$, yields
\[
 \abs{\frac{\eta+u\zeta}{\eta-u\zeta}}
 \le\frac4\delta
 \qquad(\eta\in G,\ u\zeta\in X(A,\delta)).
\]
Define $L_G$ on $X(A,\delta)$ by the same integral. If $z_n\to z$
in this set, then the kernel converges pointwise for every $\eta\in G$.
The majorant $(4/\delta)\mathbf1_G$ is integrable against the finite
measure $\mu$. Dominated convergence therefore gives $L_G(z_n)\to L_G(z)$,
including when $z\in\T$. In fact, writing $K(\eta,z)=(\eta+z)/(\eta-z)$,
the denominator bound gives the explicit estimate
\[
 |K(\eta,z)-K(\eta,w)|
 =\frac{2|z-w|}{|\eta-z|\,|\eta-w|}
 \le\frac{8}{\delta^2}|z-w|,
 \qquad z,w\in X(A,\delta),\ \eta\in G.
\]
Thus $L_G$ is uniformly continuous there, with Lipschitz constant
$8\mu(G)/\delta^2$, and $|L_G(z)|\le4\mu(G)/\delta$.
Moreover,
\[
 \operatorname{Re}K(\eta,z)=\frac{1-|z|^2}{|\eta-z|^2}\ge0;
\]
it is zero when $|z|=1$. Consequently $\operatorname{Re}L_G\ge0$ on
the entire radial set, and
\[
 \abs{1-e^{-L_G(z)}}
 =\abs{L_G(z)\int_0^1e^{-vL_G(z)}\,dv}
 \le\abs{L_G(z)}
 \le\frac4\delta\mu(G).
\]
Combining the two estimates gives exactly
\[
 |1-P_Ge^{-L_G}|
 \le |1-P_G|+|P_G|\,|1-e^{-L_G}|
 \le\frac4\delta\left(\sum_{\eta_j\in G}s_j+\mu(G)\right),
\]
which is \eqref{factor:off-estimate}.

It remains to check continuity for approaches from the disk that
need not stay in $X(A,\delta)$. Fix $\zeta_0\in X(A,\delta)\cap\T$.
The map $z\mapsto z/|z|$ is continuous near $\zeta_0$, so there is
a relative neighborhood $V\subset\overline\D\setminus\{0\}$ of
$\zeta_0$ such that $|z/|z|-\zeta_0|<\delta/2$ for $z\in V$.
Then
\[
 \operatorname{dist}(z/|z|,A)
 \ge\operatorname{dist}(\zeta_0,A)-|z/|z|-\zeta_0|
 \ge\delta/2,
\]
and hence $V\subset X(A,\delta/2)$. The product and integral
arguments already proved on this larger radial set give a continuous
extension on $V$. It agrees with the original analytic function in
$V\cap\D$ and with the first extension on $V\cap X(A,\delta)$,
because both are given by the same product and integral limits.
This proves continuity relative to $\D\cup X(A,\delta)$ and, in
particular, identifies the radial boundary limits.
In particular, restricting an arbitrary singular measure to $G$
does not enlarge its support beyond $A$: every point outside the
closed set $A$ has a neighborhood disjoint from $G$. Only the mass
$\mu(G)$ and the separation $\delta$ enter the estimates.
\end{proof}

\begin{lemma}[An exact $H^1$ norm identity]
\label{factor:lone}
Let $\vartheta$ be inner, with $\vartheta(0)=a\in(0,1)$. Then
\begin{equation}
 \norm{1-a\vartheta}_{L^1(m)}
 =(1-a^2)\int_\T\frac{dm(\xi)}{\abs{1-a\xi}}
 \le 2(1-a^2)\log\frac{e}{1-a},
 \label{factor:lone-identity}
\end{equation}
\end{lemma}

\begin{proof}
For a continuous function $H$ on $\T$, let $h$ denote its harmonic
Poisson extension to $\D$, continuously extended to the closed disk.
Positivity and unit mass of the Poisson kernel give
$|h|\le\norm H_\infty$ on the closed disk. The function
$h\circ \vartheta$ is harmonic: the chain rule gives
$\Delta(h\circ \vartheta)=|\vartheta'|^2(\Delta h)\circ \vartheta=0$,
applied separately to real and imaginary parts if $H$ is complex.
For each $0<r<1$, its mean on $r\T$ equals $h(\vartheta(0))$.
As $r\uparrow1$, the integrand converges almost everywhere to
$H(\vartheta^*)$ and is bounded by $\norm H_\infty$.
Dominated convergence therefore implies
\[
 \int_\T H(\vartheta^*(\zeta))\,dm(\zeta)
 =h(\vartheta(0))
 =\int_\T H(\xi)\frac{1-a^2}{\abs{\xi-a}^2}\,dm(\xi),
\]
where $\vartheta^*$ denotes the almost everywhere unimodular radial
boundary function.  Take $H(\xi)=\abs{1-a\xi}$ and use
$\abs{\xi-a}=\abs{1-a\xi}$ on $\T$ to obtain the equality in
\eqref{factor:lone-identity}.

For the estimate, denote the integral by $J(a)$. If $a\le1/2$,
then $J(a)\le2\le2\log(e/(1-a))$.
If $a>1/2$, put $\varepsilon=1-a$. For $0\le x\le\pi$,
\[
 \abs{1-ae^{ix}}^2=\varepsilon^2+4a\sin^2(x/2)
 \ge\varepsilon^2+\frac{2x^2}{\pi^2}
 \ge\frac{(\varepsilon+x)^2}{\pi^2}.
\]
Here $\sin(x/2)\ge x/\pi$, and the last inequality follows from
$(x-\varepsilon)^2+(\pi^2-2)\varepsilon^2\ge0$.
Symmetry and the normalization of $m$ now give
\[
 J(a)=\frac1\pi\int_0^\pi\frac{dx}{|1-ae^{ix}|}
 \le\int_0^\pi\frac{dx}{\varepsilon+x}
 =\log\frac{\varepsilon+\pi}{\varepsilon}
 \le2\log\frac e\varepsilon,
\]
where the final inequality uses $0<\varepsilon<1$ and $1+\pi<e^2$.
\end{proof}

\begin{proposition}[Separated product test functions]
\label{factor:tests}
There are universal constants $c,C>0$ with the following property.
Suppose that $\theta$ is inner,
\[
 \theta(0)=e^{-t}>0,\qquad
 0<t\le\frac12,\qquad q_\theta(0)\le\frac12.
\]
Then there is a factorization $\theta=\phi\psi\chi$ into inner
functions, with $\phi(0),\psi(0)>0$, such that
\[
 ct\le-\log\phi(0)\le t,\] and \[ct\le -\log\psi(0)\le t.
\]
The functions
\[
 f=k_0^\phi=1-\phi(0)\phi,
 \qquad
 g=k_0^\psi=1-\psi(0)\psi
\]
belong to $K_\theta\cap H^\infty$ and satisfy
\begin{equation}
 f(0)g(0)\ge ct^2,
 \qquad
 \norm{fg}_{L^1(m)}\le Ct^2\log(e/t).
 \label{factor:tests-estimates}
\end{equation}
\end{proposition}

\begin{proof}
Apply Lemmas~\ref{factor:canonical} and \ref{factor:split} to
$\nu$, and retain the resulting sets $E,F$.
Put
\[
 \phi=\theta_E,\qquad
 \psi=\theta_F,\qquad
 \chi=\theta_{\T\setminus(E\cup F)}.
\]
These three factors account for all zero occurrences and for the
entire singular measure.  They satisfy $\theta=\phi\psi\chi$.
By \eqref{factor:measure-comparison},
\eqref{factor:split-masses}, and $M\ge t$, we may take
$\kappa=C_0^{-1}\min\{1/16,\gamma\}>0$ to obtain
\[
 \kappa t\le\sigma(E)\le t,
\]
and 
\[ \kappa t\le\sigma(F)\le t.\]
Let $a=\phi(0)=e^{-\sigma(E)}$ and
$b=\psi(0)=e^{-\sigma(F)}$.  Since
\[
 2e^{-1}s\le1-e^{-2s}\le2s\qquad(0\le s\le1/2),
\]
we have $1-a^2\asymp t$ and $1-b^2\asymp t$, with universal
comparison constants.  Consequently
\[
 f(0)g(0)=(1-a^2)(1-b^2)\ge4e^{-2}\kappa^2t^2.
\]
The positive lower mass bounds imply $a,b<1$, so $f$ and $g$ are
nonzero bounded reproducing kernels in $K_\phi$ and $K_\psi$.
Since $\theta H^2\subset\phi H^2$ and $\theta H^2\subset\psi H^2$,
taking orthogonal complements gives $K_\phi,K_\psi\subset K_\theta$.

To estimate the product, set $\delta=\Delta_0/2$ and
\[
 U=\{\zeta\in\T:
             \operatorname{dist}(\zeta,\overline E)<\delta\}.
\]
This is where the separation in Lemma~\ref{factor:split} is used.
For $\zeta\in U$, the triangle inequality gives
\[
 \operatorname{dist}(\zeta,\overline F)
 \ge\operatorname{dist}(\overline E,\overline F)
       -\operatorname{dist}(\zeta,\overline E)
 >\Delta_0-\delta=\delta.
\]
For $\zeta\notin U$, by definition
$\operatorname{dist}(\zeta,\overline E)\ge\delta$.
Thus every point of the circle is separated from at least one of
the two angular supports by the same fixed distance $\delta$.
Apply Lemma~\ref{factor:off} with $(G,A)=(E,\overline E)$ on
$\T\setminus U$, and with $(G,A)=(F,\overline F)$ on $U$.
The identity $1-a\phi=(1-a)+a(1-\phi)$ and $1-e^{-s}\le s$ give
\[
 \abs{f}\le\left(1+\frac4\delta\right)\sigma(E)\le Ct
      \quad\text{on }\T\setminus U,
 \qquad
 \abs{g}\le\left(1+\frac4\delta\right)\sigma(F)\le Ct
      \quad\text{on }U.
\]
These statements use the continuous boundary values of the factor
that is evaluated away from its support; the other factor has its
usual radial boundary value almost everywhere.  Thus
\begin{equation}
 \norm{fg}_{L^1(m)}
 \le Ct\int_U\abs{f}\,dm
       +Ct\int_{\T\setminus U}\abs{g}\,dm
 \le Ct\bigl(\norm{f}_{L^1(m)}+\norm{g}_{L^1(m)}\bigr).
 \label{factor:product-split}
\end{equation}
Finally, $e^{-1/2}s\le1-e^{-s}\le s$ for $0\le s\le1/2$, so
\[
 e^{-1/2}\kappa t\le1-a\le t,\]
 and \[ e^{-1/2}\kappa t\le1-b\le t.
\]
In particular,
$\log(e/(1-a)),\log(e/(1-b))\le\log(e/t)+\log(e^{1/2}/\kappa)$.
Since $\kappa$ is fixed and $\log(e/t)\ge1$, Lemma~\ref{factor:lone}
and $1-a^2,1-b^2\le2t$ yield
\[
 \norm{f}_{L^1(m)}+\norm{g}_{L^1(m)}
 \le Ct\log(e/t).
\]
Substitution into \eqref{factor:product-split} proves
\eqref{factor:tests-estimates}.
\end{proof}

\begin{remark}
The remaining factor $\chi$ need not be constant; it enters the
argument only through the inclusions
$K_\phi,K_\psi\subset K_\theta$.
The required separation concerns only the two closed sets
$\overline E$ and $F$ used in the application of
Lemma~\ref{factor:split}.
\end{remark}

\section{The quantitative lower bound and the main theorem}\label{sec:proof}

We first prove Theorem~\ref{thm:quantitative} at the origin, and then
transfer the estimate to an arbitrary point.

\begin{proof}[Proof of Theorem~\ref{thm:quantitative}]
Suppose first that $\theta(0)=e^{-t}>0$ and $q_\theta(0)\le1/2$.
By Proposition~\ref{factor:tests}, there are inner divisors
$\phi,\psi$ and functions
\[
 f=1-\phi(0)\phi,\qquad g=1-\psi(0)\psi
\]
in $K_\theta\cap H^\infty$ with
\begin{equation}\label{proof:factor-estimates}
 f(0)g(0)\ge c_1t^2,\qquad
 \norm{fg}_1\le C_1t^2\log(e/t).
\end{equation}
The constants $c_1,C_1>0$ are absolute. Applying
Lemma~\ref{pre:pair} at the origin gives
\begin{align}
 \beta_\theta(R_0^\theta)
 &\ge\frac{f(0)g(0)}{(1-e^{-2t})\norm{fg}_1}\notag\\
 &\ge\frac{c_1}{C_1(1-e^{-2t})\log(e/t)}
 \ge\frac{c_1}{2C_1t\log(e/t)}.
 \label{proof:pairing}
\end{align}
This proves the asserted estimate when $a=0$.

For a general point $a$, set $\eta=\theta\circ\tau_a$, where
$\tau_a$ is the disk automorphism defined in Lemma~\ref{pre:unitary}.
By Lemma~\ref{pre:unitary},
\[
 \beta_\theta(R_a^\theta)=\beta_\eta(R_0^\eta),
 \qquad q_\eta(0)=q_\theta(a),
 \qquad \abs{\eta(0)}=\abs{\theta(a)}.
\]
Multiplication of $\eta$ by a unimodular constant makes
$\eta(0)=e^{-t}>0$ and leaves the infimal bounded-symbol norm unchanged.
The preceding estimate therefore applies with exactly the same
absolute constant.
\end{proof}

\begin{proof}[Proof of Corollary~\ref{cor:divergence}]
Take the points $a_n$ from Corollary~\ref{geom:sequence}. Put
$t_n=-\log\abs{\theta(a_n)}$. Then $t_n>0$ and $t_n\to0$.
After discarding finitely many terms, $t_n\le1/2$ and
$q_\theta(a_n)\le1/2$. Theorem~\ref{thm:quantitative} yields
\[
 \beta_\theta(R_{a_n}^\theta)
 \ge\frac{c}{t_n\log(e/t_n)}\longrightarrow\infty,
\]
because $t\log(e/t)\to0$ as $t\downarrow0$.
\end{proof}

\begin{proof}[Proof of Theorem~\ref{thm:main}]
Fix a nonconstant inner function $\theta$ throughout the proof.

\emph{(i)$\Rightarrow$(ii).} This is the sufficient direction proved
in \cite[Corollary~2.5]{BBK}.

\emph{(ii)$\Rightarrow$(iii).} Condition~(ii) says exactly that the
bounded symbol map $Q_\theta$ in \eqref{pre:symbol-map} is surjective.
Both its domain and its range are Banach spaces. Applying the open
mapping theorem as in Lemma~\ref{pre:uniform} gives a finite constant
$B_\theta$, independent of the operator, such that
\[
 \beta_\theta(A)\le B_\theta\norm A
 \qquad(A\in\TT(\theta)).
\]
Each $R_a^\theta$ belongs to $\TT(\theta)$ and has norm one. Hence
$\beta_\theta(R_a^\theta)\le B_\theta$ for every $a\in\D$, which is~(iii).

\emph{(iii)$\Rightarrow$(i).} Suppose instead that $\theta$ is not
one-component. Corollary~\ref{geom:sequence} provides points $a_n$
with $|\theta(a_n)|\to1$ and $q_\theta(a_n)\to0$. Put
$t_n=-\log|\theta(a_n)|$. For all sufficiently large $n$, the
hypotheses of Theorem~\ref{thm:quantitative} hold, and therefore
\[
 \beta_\theta(R_{a_n}^\theta)
 \ge\frac{c}{t_n\log(e/t_n)}\longrightarrow\infty.
\]
This contradicts~(iii), proving that $\theta$ is one-component.

In particular, for a non-one-component $\theta$ the symbol map cannot
be surjective. Thus at least one bounded operator in $\TT(\theta)$
has no bounded symbol. The divergent sequence is used to contradict
the uniform bound forced by surjectivity; each individual member of
that sequence still has a bounded symbol.
\end{proof}

\begin{proof}[Proof of Corollary~\ref{cor:traceclass}]
Write $\Sone(K_\theta)$ for the trace-class ideal, with trace norm
$\norm{\cdot}_{\Sone}$, and let
\[
 X_\theta=\TT(\theta)\cap\Sone(K_\theta).
\]
This is a Banach space in trace norm: a trace-norm Cauchy sequence
converges in $\Sone$, hence in operator norm, and $\TT(\theta)$ is
operator-norm closed.

Let $Q_\theta:L^\infty\to\TT(\theta)$ be the symbol map, and put
\[
 Z_\theta=L^\infty/\ker Q_\theta.
\]
The kernel is closed because $Q_\theta$ is bounded. Thus $Z_\theta$
is Banach, and $Q_\theta$ induces a continuous injective map
$J_\theta:Z_\theta\to\TT(\theta)$. If $A$ has a bounded symbol,
the norm of its unique preimage under $J_\theta$ is precisely
$\beta_\theta(A)$.

Suppose, to obtain a contradiction, that every operator in $X_\theta$
has a bounded symbol. The map
\[
 L:X_\theta\to Z_\theta,\qquad L(A)=J_\theta^{-1}(A)
\]
is then everywhere defined and linear. It has a closed graph:
if $A_j\to A$ in trace norm and $L(A_j)\to v$ in $Z_\theta$, then
continuity of $J_\theta$ gives
$J_\theta v=\lim_j A_j=A$ in operator norm; injectivity gives $v=L(A)$.
The closed graph theorem therefore yields
\[
 \beta_\theta(A)\le C\norm A_{\Sone},\qquad A\in X_\theta.
\]
For the sequence in Corollary~\ref{cor:divergence}, each $R_{a_n}^\theta$
has rank one and operator norm one. Its unique nonzero singular value
is one, so $\norm{R_{a_n}^\theta}_{\Sone}=1$. The displayed bound
contradicts divergence of the infimal bounded-symbol norms.
\end{proof}

\section{Sharpness in a degree-two family}\label{sec:examples}

The logarithmic loss in the denominator of
Theorem~\ref{thm:quantitative} is unavoidable.
The following degree-two family gives matching upper and lower bounds
for the infimal bounded-symbol norm and determines its precise
asymptotic order.

\begin{proposition}\label{prop:degree-two-sharpness}
For $0<r<1$, let
\[
 \theta_r(z)=\frac{r-z^2}{1-rz^2},
 \qquad t=-\log r,
\]
and define
\[
 w_r(e^{ix})=\frac{1-r^2}{\abs{1-re^{2ix}}^2},
\]
\[
 \begin{aligned}
 C_r&=\int_\T\abs{\cos x}\,w_r(e^{ix})\,dm(e^{ix}),\\
 S_r&=\int_\T\abs{\sin x}\,w_r(e^{ix})\,dm(e^{ix}).
 \end{aligned}
\]
Then each $\theta_r$ is one-component and satisfies
$\theta_r(0)=r$ and $q_{\theta_r}(0)=0$. Its infimal
bounded-symbol norm satisfies
\begin{equation}\label{example:two-sided}
 \frac{1}{2S_r}
 \le \beta_{\theta_r}(R_0^{\theta_r})
 \le \frac{1}{2S_r}
       \sqrt{1+\left(\frac{S_r}{C_r}\right)^2}.
\end{equation}
Moreover,
\begin{equation}\label{example:exact-integrals}
 S_r=\frac{1-r}{\pi\sqrt r}
       \log\frac{1+\sqrt r}{1-\sqrt r},
 \qquad
 C_r=\frac{1+r}{\pi\sqrt r}
       \arctan\frac{2\sqrt r}{1-r}.
\end{equation}
Consequently,
\begin{equation}\label{example:sharp-asymptotic}
 \beta_{\theta_r}(R_0^{\theta_r})
 \sim\frac{\pi}{2t\log(1/t)}
 \qquad(t\downarrow0).
\end{equation}
In particular, the scale $[t\log(e/t)]^{-1}$ in
Theorem~\ref{thm:quantitative} is sharp up to absolute multiplicative
constants, already among degree-two Blaschke products.
\end{proposition}

\begin{proof}
Clearly $\theta_r(0)=r$ and $\theta_r'(0)=0$, so
$q_{\theta_r}(0)=0$.
Put $s=\sqrt r$. Since
\[
 \theta_r(z)
 =-\frac{z-s}{1-sz}\frac{z+s}{1+sz},
\]
the function $\theta_r$ is a Blaschke product of degree two, with
zeros $s$ and $-s$. Thus $K_{\theta_r}$ has
dimension two. The two corresponding reproducing kernels are
$1/(1-sz)$ and $1/(1+sz)$; they are linearly independent and hence
span $K_{\theta_r}$. Their sum and difference show that
\[
 K_{\theta_r}
 =\operatorname{span}\left\{
      \frac{1}{1-rz^2},\frac{z}{1-rz^2}\right\}.
\]
The functions
\begin{equation}\label{example:orthonormal-basis}
 e_0(z)=\frac{\sqrt{1-r^2}}{1-rz^2},
 \qquad e_1(z)=ze_0(z),
\end{equation}
form an orthonormal basis. Indeed,
$\abs{e_0(e^{ix})}^2=\abs{e_1(e^{ix})}^2=w_r(e^{ix})$;
since $w_r(e^{ix})$ is the Poisson kernel of parameter $r$ evaluated
at $e^{2ix}$,
\[
 \int_\T w_r\,dm=1.
\]
Moreover,
\[
 \ip{e_0}{e_1}
 =\int_\T e^{-ix}w_r(e^{ix})\,dm(e^{ix})=0,
\]
because $w_r(e^{i(x+\pi)})=w_r(e^{ix})$ whereas
$e^{-i(x+\pi)}=-e^{-ix}$.

At the origin,
\begin{align*}
 k_0^{\theta_r}(z)
   &=\frac{1-r^2}{1-rz^2}=\sqrt{1-r^2}\,e_0(z),\\
 \widetilde k_0^{\theta_r}(z)
   &=-\frac{(1-r^2)z}{1-rz^2}
   =-\sqrt{1-r^2}\,e_1(z).
\end{align*}
Since $D_{\theta_r}(0)=1-r^2$, our tensor convention gives
\begin{equation}\label{example:rank-one-matrix}
 R_0^{\theta_r}=-e_1\otimes e_0.
\end{equation}
Thus, in the ordered basis $(e_0,e_1)$,
\[
 [R_0^{\theta_r}]
 =\begin{pmatrix}0&0\\-1&0\end{pmatrix}.
\]

For $\varphi\in L^\infty(\T)$, set
\[
 m_j(\varphi)=\int_\T
      \varphi(e^{ix})e^{ijx}w_r(e^{ix})\,dm(e^{ix}),
 \qquad j\in\mathbb Z.
\]
Using \eqref{example:orthonormal-basis} and the convention that the
inner product is linear in its first variable, the matrix of the
compression $A_\varphi^{\theta_r}$ is
\begin{equation}\label{example:moment-matrix}
 \bigl[\ip{A_\varphi^{\theta_r}e_j}{e_i}\bigr]_{i,j=0}^1
 =\begin{pmatrix}
    m_0(\varphi)&m_1(\varphi)\\
    m_{-1}(\varphi)&m_0(\varphi)
  \end{pmatrix}.
\end{equation}
Comparing \eqref{example:moment-matrix} with
\eqref{example:rank-one-matrix} shows that
$A_\varphi^{\theta_r}=R_0^{\theta_r}$ if and only if
\begin{equation}\label{example:moment-conditions}
 m_0(\varphi)=0,
 \qquad m_1(\varphi)=0,
 \qquad m_{-1}(\varphi)=-1.
\end{equation}

Let
\[
 \varepsilon_c(x)=\operatorname{sgn}(\cos x),
 \qquad \varepsilon_s(x)=\operatorname{sgn}(\sin x),
\]
where the values at the zeros are immaterial. The symmetries
$w_r(e^{-ix})=w_r(e^{ix})$ and
$w_r(e^{i(x+\pi)})=w_r(e^{ix})$ give
\begin{align*}
 \int_\T \varepsilon_c(x)w_r(e^{ix})\,dm
   &=\int_\T \varepsilon_s(x)w_r(e^{ix})\,dm=0,\\
 \int_\T \varepsilon_c(x)e^{ix}w_r(e^{ix})\,dm
   &=\int_\T \varepsilon_c(x)e^{-ix}w_r(e^{ix})\,dm=C_r,\\
 \int_\T \varepsilon_s(x)e^{ix}w_r(e^{ix})\,dm
   &=iS_r,\\
 \int_\T \varepsilon_s(x)e^{-ix}w_r(e^{ix})\,dm
   &=-iS_r.
\end{align*}
Since $C_r>0$ and $S_r>0$, the following formula is
well-defined:
\begin{equation}\label{example:bounded-symbol}
 \omega_r(e^{ix})
 =-\frac{\varepsilon_c(x)}{2C_r}
   -i\frac{\varepsilon_s(x)}{2S_r}.
\end{equation}
The preceding four identities give
\begin{align*}
 m_0(\omega_r)&=0,\\
 m_1(\omega_r)&=-\frac{C_r}{2C_r}
                 -i\frac{iS_r}{2S_r}=0,\\
 m_{-1}(\omega_r)&=-\frac{C_r}{2C_r}
                 -i\frac{-iS_r}{2S_r}=-1.
\end{align*}
Thus \eqref{example:moment-conditions} shows that $\omega_r$ is a
bounded symbol for $R_0^{\theta_r}$. Since
$\abs{\varepsilon_c(x)}=\abs{\varepsilon_s(x)}=1$ almost everywhere,
\[
 \norm{\omega_r}_\infty
 =\sqrt{\frac{1}{4C_r^2}+\frac{1}{4S_r^2}}
 =\frac{1}{2S_r}\sqrt{1+\left(\frac{S_r}{C_r}\right)^2}.
\]
This proves the upper bound in \eqref{example:two-sided}.

Conversely, every bounded symbol $\varphi$ representing
$R_0^{\theta_r}$ satisfies \eqref{example:moment-conditions}. Hence
\[
 -1=m_{-1}(\varphi)-m_1(\varphi)
 =\int_\T\varphi(e^{ix})(e^{-ix}-e^{ix})
       w_r(e^{ix})\,dm(e^{ix}),
\]
and therefore
\[
 1\le 2\norm{\varphi}_\infty
          \int_\T\abs{\sin x}\,w_r(e^{ix})\,dm(e^{ix})
   =2S_r\norm{\varphi}_\infty.
\]
Taking the infimum over all representing bounded symbols proves the
lower bound in \eqref{example:two-sided}.

It remains to evaluate the two integrals. Since $w_r$ is even and
$\pi$-periodic, each integral over $\T$ equals four times its integral
over $[0,\pi/2]$; the normalization $dm=dx/(2\pi)$ then gives the
factor $2/\pi$ below. Moreover,
\[
 \abs{1-re^{2ix}}^2
 =(1-r)^2+4r\sin^2x
 =(1+r)^2-4r\cos^2x.
\]
Thus, after the substitutions $u=\sin x$ and $u=\cos x$, respectively,
\begin{align*}
 C_r
 &=\frac{2(1-r^2)}{\pi}
   \int_0^{\pi/2}
     \frac{\cos x\,dx}{(1-r)^2+4r\sin^2x}
 =\frac{1+r}{\pi\sqrt r}
   \arctan\frac{2\sqrt r}{1-r},\\
 S_r
 &=\frac{2(1-r^2)}{\pi}
   \int_0^{\pi/2}
     \frac{\sin x\,dx}{(1+r)^2-4r\cos^2x}\\
 &=\frac{1-r}{\pi\sqrt r}
   \operatorname{artanh}\frac{2\sqrt r}{1+r}
 =\frac{1-r}{\pi\sqrt r}
   \log\frac{1+\sqrt r}{1-\sqrt r}.
\end{align*}
This proves \eqref{example:exact-integrals}. If $r=e^{-t}$ and
$t\downarrow0$, then $1-r\sim t$ and
\[
 \log\frac{1+\sqrt r}{1-\sqrt r}
 =\log\coth\frac{t}{4}\sim\log\frac1t.
\]
Therefore
\[
 C_r\longrightarrow1,
 \qquad
 S_r\sim\frac{t}{\pi}\log\frac1t,
 \qquad
 \frac{S_r}{C_r}\longrightarrow0.
\]
The two bounds in \eqref{example:two-sided} now give
\eqref{example:sharp-asymptotic}.
For all sufficiently small $t>0$, one has $t\le1/2$, while
$q_{\theta_r}(0)=0$ and
$\log(e/t)\sim\log(1/t)$. Thus this family satisfies all the
hypotheses of Theorem~\ref{thm:quantitative}, and its infimal
bounded-symbol norm has the same growth rate as the universal lower
bound. This proves the asserted sharpness.

Finally, every fixed $\theta_r$ is one-component. Indeed, if
\[
 \tau_r(w)=\frac{r-w}{1-rw},
\]
then $\theta_r(z)=\tau_r(z^2)$. For any
$\varepsilon\in(r,1)$, the pseudohyperbolic disk
\[
 G_{r,\varepsilon}
 =\{w\in\D:\abs{\tau_r(w)}<\varepsilon\}
\]
is the Euclidean disk with center and radius
\[
 c_{r,\varepsilon}
 =\frac{r(1-\varepsilon^2)}{1-r^2\varepsilon^2},
 \qquad
 \rho_{r,\varepsilon}
 =\frac{\varepsilon(1-r^2)}{1-r^2\varepsilon^2}.
\]
Since
\[
 \rho_{r,\varepsilon}-c_{r,\varepsilon}
 =\frac{(\varepsilon-r)(1+r\varepsilon)}
        {1-r^2\varepsilon^2}>0,
\]
this disk contains $0$ and hence is star-shaped about $0$. Its preimage
under the map $z\mapsto z^2$ is also star-shaped about $0$: if
$z^2\in G_{r,\varepsilon}$ and
$0\le\lambda\le1$, then $\lambda^2z^2\in G_{r,\varepsilon}$, so
$\lambda z$ lies in its inverse image. The inverse image is therefore
connected. Thus $\theta_r$ is one-component. At the same time,
\eqref{example:sharp-asymptotic} shows that the constants $B_{\theta_r}$
in Lemma~\ref{pre:uniform} cannot remain bounded as $r\uparrow1$.
\end{proof}

\section*{Acknowledgements}

This work was partially supported by the National Natural Science
Foundation of China (Grant No.~12171484 and 12571140),the Science and Technology Innovation Program of Hunan Province (Grant No. 2023RC3028), the Central South University Innovation-Driven Research Programme (Grant No. 2023CXQD032) and the Hunan
Basic Science Research Center for Mathematical Analysis
(Grant No.~2024JC2002).

\section*{AI Declaration}

The author used OpenAI's ChatGPT (GPT-5.6 Sol)  for language editing, LaTeXformatting, and limited assistance with local mathematical reasoning. The proof strategy andmathematical development are the author’ own; all mathematical content was independentlywritten and checked by the author, who take full responsibility for it.

\end{document}